\documentclass[10pt]{amsart}
\usepackage{amsmath,amsthm,amssymb,mathtools,amsfonts,mathrsfs}
\usepackage{xspace,xcolor}
\usepackage[breaklinks,colorlinks,citecolor=teal,linkcolor=teal,urlcolor=teal,pagebackref,hyperindex]{hyperref}
\usepackage[alphabetic]{amsrefs}
\usepackage{tikz-cd}
\usepackage[all]{xy}
\usepackage{bbm}
\usepackage{verbatim}

\newtheorem{thm}{Theorem}[section]

\newtheorem{conj}[thm]{Conjecture}

\theoremstyle{definition}
\newtheorem{defn}[thm]{Definition}
\newtheorem{ex}[thm]{Example}
\newtheorem{rmk}[thm]{Remark}

\newtheorem{assu}{Assumption}

\newtheorem*{nota*}{Notation}

\newtheorem{lemma}{Lemma}

\newcommand{\sO}{\mathcal O}

\newcommand{\bM}{\overline{\mathcal M}}

\newcommand{\on}{\operatorname}

\title{A mirror theorem for multi-root stacks and applications}

\author{Hsian-Hua Tseng}
\address{Department of Mathematics\\ Ohio State University\\ 100 Math Tower, 231 West 18th Ave.\\Columbus\\ OH 43210\\ USA}
\email{hhtseng@math.ohio-state.edu}

\author{Fenglong You}
\address{Department of Mathematics \\ ETH Z\"urich, \\Rämistrasse 101, \\8092 Zürich, \\Switzerland}
\email{fenglong.you@math.ethz.ch}

\begin{document}


\maketitle

\begin{abstract}
Let $X$ be a smooth projective variety with a simple normal crossing divisor $D:=D_1+D_2+...+D_n$, where $D_i\subset X$ are smooth, irreducible and nef. We prove a mirror theorem for multi-root stacks $X_{D,\vec r}$ by constructing an $I$-function lying in a slice of Givental's Lagrangian cone for Gromov--Witten theory of multi-root stacks. We provide three applications: (1) We show that some genus zero invariants of $X_{D,\vec r}$ stabilize for sufficiently large $\vec r$. (2) We state a generalized local-log-orbifold principle conjecture and prove a version of it. (3) We show that regularized quantum periods of Fano varieties coincide with classical periods of the mirror Landau--Ginzburg potentials using orbifold invariants of $X_{D,\vec r}$.
\end{abstract}

\tableofcontents

\section{Introduction}

A mirror theorem refers to a relation between a generating function of genus zero Gromov--Witten invariants (the $J$-function) and a period integral (the $I$-function) of the mirror. Such a mirror theorem was first proved by Givental \cite{Givental96} and Lian--Liu--Yau \cite{LLY}, where the $J$-function and the $I$-function are equal after a change of variables called the mirror map. The $J$-function naturally lies in Givental's Lagrangian cone.  A more general formulation of a mirror theorem is to construct an explicit $I$-function and prove that the $I$-function lies in Givental's Lagrangian cone. In this paper, we study the genus zero orbifold Gromov--Witten theory of multi-root stacks. We generalize the main theorem of \cite{FTY} to simple normal crossing divisors. In other words, we prove a mirror theorem for multi-root stacks by constructing an $I$-function which lies in Givental's Lagrangian cone.

Let $X$ be a smooth projective variety over $\mathbb{C}$. Let $$D_1,..., D_n\subset X$$ be divisors which are smooth, irreducible, and nef. For natural numbers $r_1,...,r_n$ which are pairwise co-prime, the associated multi-root stack
$$X_{(D_1, r_1), (D_2, r_2), ..., (D_n, r_n)}$$
is nonsingular and has a well-defined Gromov-Witten theory. Let 
\[
\vec r=(r_1,r_2,\ldots, r_n).
\]
We write
\[
X_{D,\vec r}:=X_{(D_1, r_1), (D_2, r_2), ..., (D_n, r_n)}.
\]
A description of $X_{D,\vec r}$ as a complete intersection inside a toric stack bundle is given in Section \ref{sec:construction}. Using this description and known results we derive a mirror theorem for $X_{D,\vec r}$ in Section \ref{sec:mir_thm}. 

There are some applications: 
\begin{itemize}
    
\item 

The first application is to take the large $\vec r$ limit of the $I$-function, which implies that relevant genus zero invariants stabilize as the $r_i$'s become sufficiently large.

\item

The second application is to formulate the local-log-orbifold principle and prove a version of it using $I$-functions. This provides a simple point of view for the local-log principle from mirror symmetry.

\item The third application is to prove that regularized quantum periods for Fano varieties coincide with generating functions of orbifold Gromov--Witten invariants of root stacks which can be viewed as classical periods of the Landau--Ginzburg superpotentials. This connects our theory with the Fano search program and the Gross--Siebert program.
\end{itemize}

\subsection{Mirror theorem and the large $\vec r$ limit}

\begin{thm}\label{thm-mirror-intro}
Let $X$ be a smooth projective variety. Let $D:=D_1+D_2+...+D_n$ be a simple normal-crossing divisor with $D_i\subset X$ smooth, irreducible and nef. The $I$-function $I_{X_{D,\vec r}}$ of the root stack $X_{D,\vec r}$ lies in Givental's Lagrangian cone $\mathcal L_{X_{D,\vec r}}$ of $X_{D,\vec r}$.
\end{thm}

Theorem \ref{thm-mirror-intro} is stated for the non-extended $I$-function in Theorem \ref{thm-mirror} and is stated for the extended $I$-function in Theorem \ref{thm:mirror-extended}, where the non-extended $I$-function and the extended $I$-function are defined in (\ref{I-orb}) and (\ref{Extended-I-function}) respectively. The extended $I$-function comes from Jiang's construction \cite{Jiang} of torick stack (and toric stack bundles) using $S$-extended stacky fan instead of stacky fans. The extended $I$-functions encode some additional orbifold data. Under some assumptions (for example when the mirror maps are trivial), the extended $I$-functions can be used to compute orbifold invariants with several orbifold markings.

Theorem \ref{thm-mirror-intro} can be used to compute the genus zero invariants of $X_{D,\vec r}$ and shows that relevant invariants stabilize when the $r_i$ are sufficiently large. Therefore, the following conjecture is true for  invariants in the $J$-function of $X_{D,\vec r}$.

\begin{conj}\label{conj-large-r}
Genus zero orbifold Gromov--Witten invariants of the root stack $X_{D,\vec r}$ (after multiplying by suitable powers of $r_i$) stabilize when $r_i$'s are sufficiently large. Moreover, higher genus orbifold  Gromov--Witten invariants of the root stack $X_{D,\vec r}$ (after multiplying by suitable powers of $r_i$) are polynomials in $r_i$ with degree bounded by $2g-1$ when the $r_i$'s are sufficiently large. 
\end{conj}

When Conjecture \ref{conj-large-r} holds, we {\em formally} consider the constant terms of the polynomials as Gromov--Witten invariants of the {\em infinitely root stacks} $$X_{D,\infty}.$$ When $D$ is a smooth divisor, results in \cite{ACW}, \cite{TY}, \cite{FWY} and \cite{FWY19} show that the formal Gromov--Witten theory of the infinitely root stacks is simply the relative Gromov--Witten theory of $(X,D)$. When $D$ is an effective reduced simple normal crossing divisor, one may expect that the formal Gromov--Witten theory of the infinitely root stacks is the logarithmic Gromov--Witten theory of $(X,D)$. However, it may not be true in general. This is why we call the large $\vec r$ limit of the Gromov--Witten theory of $X_{D,\vec r}$ the Gromov--Witten theory of the infinitely root stack $X_{D,\infty}$, instead of conjecturing the limit to be the log Gromov--Witten theory of $(X,D)$.
Furthermore, inspired by \cite{TY20}, we also expect that the degrees (in $r_i$, for each $i$) of the polynomials for the orbifold Gromov--Witten invariants of $X_{D,\vec r}$ are bounded by $2g-1$. 

Conjecture \ref{conj-large-r} has recently been proved in \cite{TY20c} and the foundation for the formal Gromov--Witten theory of $X_{D,\infty}$ has been studied there.

\begin{rmk}
Infinitely root stacks have been studied in \cite{TV}. However, the Gromov--Witten theory of infinite root stacks has not been defined. Although it is likely to be true, we do not claim that the Gromov--Witten theory of infinite root stacks (if it can be defined) is the same the limit of the Gromov--Witten theory of finite root stacks. By \cite{TV}, the infinite root stack determines the logarithmic structure. It would be interesting to define Gromov--Witten theory of infinite root stacks directly, then compare it with logarithmic Gromov--Witten theory.
\end{rmk}

In Section \ref{sec:limit}, we compute some invariants of $X_{D,\infty}$ that coincide with enumerative expectations. In Example \ref{ex-p-2-orb-inv}, when $X=\mathbb P^2$ and $D$ is the union of a line and a conic, we confirm that formal invariants of $X_{D,\infty}$ count the numbers of curves in $\mathbb P^2$ through one generic point and with maximal tangency to the line and the conic at one point respectively. The number is $\frac{(2d)!}{(d!)^2}$. When $X=\mathbb P^1\times \mathbb P^1$ and $D$ is the union of two distinct $(1,1)$ curves $L_1$ and $L_2$, in Example \ref{ex:p1-p1}, we confirm that formal invariants of $X_{D,\infty}$ count the numbers of curves in $\mathbb P^1\times \mathbb P^1$ through one generic point and with maximal tangency to $L_1$ and $L_2$ at one point respectively. The number is $\frac{(d_1+d_2)!^2}{(d_1!)^2(d_2!)^2}$.

\begin{rmk}
In general, orbifold invariants will not be the same as log invariants. When orbifold invariants equal to log invariants, we can use Theorem \ref{thm-mirror-intro} to compute log invariants which are usually difficult to compute. On the other hand, when orbifold invariants and log invariants are different, orbifold invariants provide another virtual count of numbers of curves with tangency conditions along a simple normal crossing divisor. These orbifold invariants are more accessible in terms of computation. Therefore, computing orbifold invariants of root stacks are interesting either way.
\end{rmk}

\subsection{The local-log-orbifold principle}

\subsubsection{Smooth divisors}
Let $X$ be a smooth projective variety. Let $D$ be a divisor that is smooth, effective, and nef. Let $\beta$ be a curve class of $X$ such that $D\cdot \beta>0$. In this case, the moduli space $$\bM_{g,l}(\sO_X(-D),\beta)$$ of genus $g$ stable maps of class $\beta$ to the total space of $\sO_X(-D)$ naturally coincides with the moduli space $$\bM_{g,l}(X,\beta)$$ of genus $g$ stable maps of class $\beta$ to $X$. The moduli space $$\bM_{g,l,(d)}(X/D, \beta)$$ of genus $g$ relative stable maps of class $\beta$ to $(X,D)$ with only one contact condition of maximal tangency along $D$ admits a natural map
$$F:\bM_{g,l,(d)}(X/D,\beta)\to \bM_{g,l}(X,\beta)$$ obtained by forgetting the relative marked point and stabilizing. 

The first instance of the {\em log-local principle} is the following equality between virtual fundamental classes of these moduli spaces, proven in \cite{vGR}:
\begin{equation}\label{eqn:log_local}
   [\bM_{0,0}(\sO_X(-D),\beta)]^{\on{vir}}=\frac{(-1)^{d-1}}{d}F_*[\bM_{0,0,(d)}(X/D,\beta)]^{\on{vir}},
\end{equation}
where $d=D\cdot \beta$. This formula was first conjectured by Takahashi \cite{Takahashi} for $\mathbb P^2$ with a smooth cubic. Takahashi's conjecture was proved by Gathmann in \cite{Gathmann}.

We formulate Equality (\ref{eqn:log_local}) in a slightly more general form as follows:
\begin{equation}\label{eqn:log_local_1}
   \on{ev}_1^*(D)\cap [\bM_{0,1}(\sO_X(-D),\beta)]^{\on{vir}}=(-1)^{d-1}F_*[\bM_{0,0,(d)}(X/D,\beta)]^{\on{vir}},
\end{equation}
where
\[
F:\bM_{0,0,(d)}(X/D,\beta)\to \bM_{0,1}(X,\beta)
\]
is the forgetful map that forgets the relative condition, but remembers the marking.
Equality (\ref{eqn:log_local}) can be recovered using the divisor equation. Note that (\ref{eqn:log_local_1}) can be proved following the proof of (\ref{eqn:log_local}) with minor adjustment. It was also pointed out by Fan--Wu in \cite{FW}.

In Section \ref{sec:sm_div}, we use the mirror theorem for relative Gromov-Witten theory of $(X,D)$, derived in \cite{FTY}, to calculate the relative Gromov-Witten invariants on the right-hand side of (\ref{eqn:log_local_1}). The local Gromov-Witten invariants on the left-hand side of (\ref{eqn:log_local_1}) is calculated using a well-known mirror theorem, see, for example, \cite{g}. More specifically, identifying the non-extended $I$-function of the relative Gromov--Witten theory of $(X,D)$ with the $I$-function of the local Gromov--Witten theory of $\mathcal O_X(-D)$ yields (\ref{eqn:log_local_1}) at the level of invariants with some extra markings.
\begin{thm}\label{thm-local-rel}
The following identity between relative and local Gromov--Witten invariants holds: 
\[
\left\langle \prod_{i=1}^l[\gamma_i]_0, [\iota^*\gamma]_{d}\bar{\psi}^a\right\rangle_{0,l,(d),\beta}^{(X,D)}=(-1)^{d-1}\left\langle \prod_{i=1}^l\gamma_i, D\cdot \gamma \bar{\psi}^a \right\rangle_{0,l+1,\beta}^{\mathcal O_{X}(-D)},  
\]
where $\gamma, \gamma_i\in H^*(X)$, $a\in \mathbb Z_{\geq 0}$ and $\iota: D \hookrightarrow X$ is the inclusion map.
\end{thm}

Using the extended $I$-function of the relative Gromov--Witten theory of $(X,D)$ and the $I$-function of the local Gromov--Witten theory of $\mathcal O_X(-D)$, we are also able to compute relevant relative invariants that appear in (\ref{eqn:log_local}). This computation is done in Section \ref{sec:extend-rel-I}. Therefore, the relative mirror theorem implies the log-local principle (\ref{eqn:log_local}) at the level of invariants.

\begin{rmk}
One can also simply understand the log-local principle from the point of view of mirror symmetry. It is not hard to notice that the non-extended relative $I$-function and the local $I$-function are almost identical. By identifying the $I$-functions, we may view local mirror symmetry as a sector (a sub-theory) of relative mirror symmetry. The part of relative mirror symmetry that corresponds to local mirror symmetry is probably the part that has been studied the most. Likewise, one may consider genus zero local Gromov--Witten theory as a sub-theory of genus zero relative Gromov--Witten theory.
\end{rmk}

\subsubsection{Normal crossing divisors}
More generally, there is a conjectural log-local principle in the simple normal crossing case. We assume instead that $$D=D_1+...+D_n$$ is an effective reduced simple normal crossing divisor with each component $D_i$ smooth, irreducible, and nef. One can consider the moduli space $\bM_{0,0,(d_1),\ldots, (d_n)}(X/D,\beta)$ of genus zero basic stable log maps of class $\beta$ to $(X,D)$ where there is one relative marking with maximal contact order $d_i$ to each component $D_i$. Then the following is conjectured in \cite{vGR}*{Conjecture 1.4}:

\begin{conj}\label{conj:log_local}
Let $\beta$ be a curve class of $X$ with $d_i:=D_i\cdot \beta>0$ for $i\in \{1,\ldots, n\}$. Then
\begin{equation}
   [\bM_{0,0}(\oplus_{i=1}^n\sO_X(-D_i),\beta)]^{\on{vir}}=\left(\prod_{i=1}^n\frac{(-1)^{d_i-1}}{d_i}\right)F_*[\bM_{0,0,(d_1),\ldots, (d_n)}(X/D,\beta)]^{\on{vir}}.
\end{equation}
\end{conj}

Conjecture \ref{conj:log_local} has been proved in some cases in \cite{BBv} and \cite{NR}.

At the level of invariants, Conjecture \ref{conj:log_local} states that, after dividing by $\prod_{i=1}^n(-1)^{d_i+1}d_i$, the genus $0$ log Gromov-Witten invariants of maximal tangency and class $\beta$ of $(X, D)$ are equal to the genus $0$ local Gromov-Witten invariants of class $\beta$ of the total space $\oplus_{i=1}^n\sO_X(-D_i)$, with the same insertions.

Inspired by Equality (\ref{eqn:log_local_1}), it is natural to formulate the following generalized conjecture. Consider a partition of the index set
\[
\{1,2,\ldots, n\}
\]
into disjoint subsets $I_1, \ldots, I_m$. 
We assume that the intersection $\cap_{i\in I_j}D_i$ is not empty for all $j\in \{1,\ldots, m\}$. Let $$\bM_{0,0,\{(d_i)\}_{i\in I_1},\ldots, \{(d_i)\}_{i \in I_m}}(X/D,\beta)$$ be the moduli space of basic stable log maps with $m$ marked points such that the $j$-th marking has maximal contact with divisors $D_i$ for all $i\in I_j$. 
Note that the $j$-th marking maps to the intersection $\cap_{i\in I_j}D_i$. 

\begin{conj}\label{conj-snc}
Let $\beta$ be a curve class of $X$ with $d_i:=D_i\cdot \beta>0$ for $i\in \{1,\ldots, n\}$. 
The following identity is true:
\begin{align}\label{eqn-conj-snc}
   &\left(\cup_{j=1}^m \on{ev}_j^*(\cup_{i\in I_j}D_i)\right)\cap\left[\bM_{0,m}(\oplus_{i=1}^n\sO_X(-D_i),\beta)\right]^{\on{vir}}\\
   \notag =&\left(\prod_{i=1}^n(-1)^{d_i-1}\right)F_*[\bM_{0,0,\{(d_i)\}_{i\in I_1},\ldots, \{(d_i)\}_{i \in I_m}}(X/D,\beta)]^{\on{vir}}.
\end{align}
\end{conj}

We can allow more interior markings on both sides of Equation (\ref{eqn-conj-snc}) as well. Note that the original conjecture of \cite{vGR} is a special case of our conjecture when $|I_j|=1$ for all $j$. 
The local mirror theorem relevant to this conjecture is also well-known. However there is no known mirror theorem for log Gromov-Witten invariants. Furthermore, Dhruv Ranganathan and Navid Nabijou recently discovered some counter-examples which show that neither Conjecture \ref{conj-snc} nor Conjecture \ref{conj:log_local} is true in fully generality.

Instead of considering log invariants, we will consider it from a different perspective. We conjecture the following relation between local invariants and orbifold invariants.
\begin{conj}\label{conj-local-orb}
Let $\beta$ be a curve class of $X$ with $d_i:=D_i\cdot \beta>0$ for $i\in \{1,\ldots, n\}$. The following identity is true:
\begin{align}
   &\left(\cup_{j=1}^m \on{ev}_j^*(\cup_{i\in I_j}D_i)\right)\cap\left[\bM_{0,m}(\oplus_{i=1}^n\sO_X(-D_i),\beta)\right]^{\on{vir}}\\
   \notag =&\left(\prod_{i=1}^n(-1)^{d_i-1}\right)F_*[\bM_{0,0,\{(d_i)\}_{i\in I_1},\ldots, \{(d_i)\}_{i \in I_m}}(X_{D,\vec r},\beta)]^{\on{vir}},
\end{align}
when the $r_i$'s are sufficiently large.
\end{conj}

Conjecture \ref{conj-local-orb} can be proved at the level of invariants for a special case by explicit computation on both sides via mirror theorems. Using the non-extended $I$-function of $X_{D,\infty}$, we have the following result when $m=1$.

\begin{thm}\label{thm-local-orb}
Suppose  the intersection of the divisors $\cap_{i=1}^n D_i$ is not empty. Let $\beta$ be a curve class of $X$ with $d_i:=D_i\cdot \beta>0$ for $i\in \{1,\ldots, n\}$. Then the following equality holds:
\[
\left\langle \prod_{i=1}^l[\gamma_i]_0, [\iota^*\gamma]_{\vec d}\bar{\psi}^a\right\rangle_{0,l,(\vec d),\beta}^{X_{D,\infty}}=\left(\prod_{i=1}^n(-1)^{d_i-1}\right)\left\langle \prod_{i=1}^l \gamma_i, (\cup_{i=1}^n D_i)\cdot \gamma \bar{\psi}^a \right\rangle_{0,l+1,\beta}^{\oplus_{i=1}^n\sO_X(-D_i)}.  
\]
\end{thm}

Using the extended $I$-function, we have the following identity which is a slight generalization of Conjecture \ref{conj:log_local} at the level of invariants and replace log invariants by orbifold invariants.
\begin{thm}\label{thm-local-orb-ext}
Let $\beta$ be a curve class of $X$ with $d_i:=D_i\cdot \beta>0$ for $i\in \{1,\ldots, n\}$. Then
the following identity is true:
\begin{align*}
&\left\langle\prod_{i=1}^l [\gamma_i]_0, \prod_{i=1}^n [\textbf 1]_{(0,\ldots,0,d_i,0,\ldots,0)}, [\gamma]_{0}\bar{\psi}^a\right\rangle_{0,l+1,(d_1),\ldots, (d_n),\beta}^{X_{D,\infty}}\\
=&\left(\prod_{i=1}^n (-1)^{d_i-1}\right)\left\langle \prod_{i=1}^l \gamma_i, \prod_{i=1}^nD_i,\gamma\bar{\psi}^a \right\rangle_{0,n+l+1,\beta}^{\bigoplus_{i=1}^n \mathcal O_X(-D_i)},  
\end{align*}

where $\gamma,\gamma_i\in H^*(X)$ for $i=1,\ldots, l$.
\end{thm}

It is natural to believe that one should replace Conjecture \ref{conj:log_local} and Conjecture \ref{conj-snc} by Conjecture \ref{conj-local-orb}. 

 \begin{rmk}
 Our result suggests a close relationship between the theory obtained from the large $r_i$ limit of the Gromov--Witten theory of $X_{D,\vec r}$ and log Gromov--Witten theory of $(X,D)$, which would generalize the main results of \cite{TY}. However, as pointed out by Dhruv Ranganathan, orbifold invariants and log invariants are not the same in general. It would be interesting to find their precise relation and determine when they will coincide.
 
 In the case that the root stack invariants equal to the log invariants, Conjecture \ref{conj-local-orb} implies Conjecture \ref{conj-snc}. Then results in \cite{BBv} and \cite{NR} would be special cases of our result. First of all, we do not require $X$ to be toric. Secondly, $D$ is not necessary $-K_X$ or a toric divisor. Last but not least, we put the divisor classes $D_i$ as insertions of local Gromov--Witten invariants of $\bigoplus_{i=1}^n \mathcal O_X(-D_i)$ instead of removing them using divisor equations. This allows descendant classes which were not allowed previously in \cite{vGR} unless the class $\gamma$ is Poincar\'e dual of a cycle class which does not meet $D_i$. 
\end{rmk}

\begin{rmk}
One can obtain more general results about the relation between local invariants $\bigoplus_{i=1}^n \mathcal O_X(-D_i)$ and invariants of $X_{D,\infty}$ by manipulating their $I$-functions. However, identification between their $I$-functions can already be viewed as identifying local Gromov--Witten theory of $\bigoplus_{i=1}^n \mathcal O_X(-D_i)$ as a sub-theory of Gromov--Witten theory of $X_{D,\infty}$.
\end{rmk}

\subsection{Quantum periods and classical periods}

The Fano search program studies a new approach to the classification of Fano manifolds by studying their mirror Landau--Ginzburg models. The quantum period $G_X$ of a Fano variety $X$ is a generating function for certain Gromov--Witten invariants of $X$ which plays an important role in the Fano search program. Mirror symmetry for Fano varieties suggests an equivalence between the regularized quantum period $\hat{G}_X$ of a Fano variety and the classical period $\pi_W$ of its mirror Landau--Ginzburg potential $W$.

The Frobenius structure conjecture of \cite{GHK} and the construction of \cite{CPS} suggest a precise way of constructing the Landau--Ginzburg potential $W$. A classical period $\pi_W$ can be defined in terms of the constant terms of powers of $W$. A Landau--Ginzburg potential $W$ can be said to be a mirror of a Fano variety $X$ if their respective regularized quantum period and the classical period coincide:
\[
\hat{G}_X=\pi_W.
\]

We recalled that relative quantum coomology of a pair $(X,D)$ provides a ring structure to the state space (ring of insertions) of the Gromov--Witten theory of $(X,D)$. Similar to quantum cohomology for absolute Gromov--Witten theory, the relative quantum product is given by genus zero Gromov--Witten invariants of the pair $(X,D)$. According to the Frobenius structure conjecture \cite{GHK}, the Landau--Ginzburg potential $W$ is defined as
\begin{align}\label{W-theta}
W:=\vartheta_{[D_1]}+\cdots+\vartheta_{[D_n]},
\end{align}
where $D_1,\ldots, D_n$ are irreducible components of $D\in |-K_X|$ and the $\vartheta$'s are theta functions which form a canonical basis of $QH^0_{\log}(X,D)$--the degree $0$ subalgebra of the relative quantum cohomology ring $QH^*_{\log}(X,D)$.
Since the superpotential is given by the theta functions and the theta functions are defined in terms of log Gromov--Witten invariants in intrinsic mirror symmetry, coefficients of the classical period of the superpotential $W$ are written in terms of log Gromov--Witten invariants. We refer to Section \ref{sec:frobenius} for the precise definition of the classical periods. We compute the corresponding orbifold Gromov--Witten invariants of $X_{D,\infty}$ instead of log Gromov--Witten invariants and prove the following mirror equivalence.

\begin{thm}\label{thm:period}
Given a Fano variety $X$ and a divisor $D\in |-K_X|$ satisfying Assumption \ref{assumption}, the regularized quantum period and the classical period coincide 
\[
\hat{G}_X=\pi_W,
\]
where we replace the relevant log Gromov--Witten invariants with the corresponding orbifold Gromov--Witten invariants of $X_{D,\infty}$.
\end{thm}

The relation between the classical period $\pi_W$ and the classical period of a Laurent polynomial $f$ is explained in \cite{Mandel}*{Section 1.4.1}. One can try to obtain a Laurent polynomial $f$ from the superpotential $W$. According to the Fano search program, Fano varieties might be classified via their mirror Laurent polynomials (up to mutation equivalence). 

\begin{rmk}
In \cite{TY20c}*{Section 7}, we define the degree zero part of the relative quantum cohomology using the formal Gromov--Witten theory of infinite root stacks and use it to construct mirrors following the Gross--Siebert program \cite{GS}. The precise relationship between our mirror construction and Gross--Siebert's construction is not known. However, Theorem \ref{thm:period} provides an evidence that two constructions may be closely related at least when the divisor $D$ is sufficiently degenerated.
\end{rmk}

\subsection{Acknowledgment}

We would like to thank Qile Chen, Michel van Garrel, Mark Gross, Travis Mandel, Dhruv Ranganathan, Matthew Satriano and Mattia Talpo for helpful discussions and comments. H.-H. T. is supported in part by Simons foundation collaboration grant.  F. Y. is supported by a postdoctoral fellowship funded by NSERC and Department of Mathematical Sciences at the University of Alberta.

\section{Preliminary on orbifold Gromov--Witten theory}
In this section, we briefly review the definition of orbifold Gromov--Witten invariants and Givental's formalism. We refer readers to  \cite{AV}, \cite{AGV02}, \cite{AGV}, \cite{CR} and \cite{Tseng} for the foundation of orbifold Gromov--Witten theory.

Let $\mathcal X$ be a smooth proper Deligne--Mumford stack whose coarse moduli space $X$ is projective. Let $\bM_{0,n}(\mathcal X, \beta)$ be the moduli stack of $l$-pointed genus-zero degree $\beta$ stable maps to $\mathcal X$ with sections to gerbes at the markings (see \cite{AGV}*{Section 4.5} and \cite{Tseng}*{Section 2.4}). Recall that the domain curves of a stable map to  a stack can be orbicurves. In other words, the domain curve can have nontrivial stack structures at marked points and nodes. The stable maps are required to respect the stack structures of the domain and the target. The natural evaluation maps land on the inertia stack $I \mathcal X:=\mathcal X\times_{\Delta,\mathcal X\times \mathcal X,\Delta}\mathcal X$, where $\Delta: \mathcal X\rightarrow \mathcal X \times \mathcal X$. The Chen--Ruan orbifold cohomology $H^*_{\on{CR}}(\mathcal X)$ of $\mathcal X$ is the cohomology of the inertia stack $I\mathcal X$ with degree shifted by ages. 
The genus-zero orbifold Gromov--Witten invariants of $\mathcal X$ are defined as follows
\begin{align}
\left\langle \prod_{i=1}^l \tau_{a_i}(\gamma_i)\right\rangle_{0,l,\beta}^{\mathcal X}:=\int_{[\bM_{0,l}(\mathcal X, \beta)]^{w}}\prod_{i=1}^l(\on{ev}^*_i\gamma_i)\bar{\psi}_i^{a_i},
\end{align}
where 
\begin{itemize}
\item  $\gamma_i\in H_{\on{CR}}^*(\mathcal X)$ are cohomological classes.
\item $a_i$ are non-negative integers, for $1\leq i\leq l$.
\item $[\bM_{0,l}(\mathcal X, \beta)]^{w}$ is the the weighted virtual fundamental class in \cite{AGV02}*{Section 4.6} and \cite{Tseng}*{Section 2.5.1}. 
    \item 
For $i=1,2,\ldots,l$,
\[
\on{ev}_i: \bM_{0,l}(\mathcal X,\beta) \rightarrow I\mathcal X
\]
is the evaluation map.
\item
$\bar{\psi}_i\in H^2(\bM_{0,l}(\mathcal X, \beta),\mathbb Q)$
is the descendant class.
\end{itemize}

Let $t_i=\sum_{\alpha} t_{i;\alpha}\phi_\alpha\in H^*_{\on{CR}}(\mathcal X)$, where $t_{i;\alpha}$ are formal variables and
\[
\{\phi_\alpha\}\subset H^*_{\on{CR}}(\mathcal X)
\]
is an additive basis.
The genus-zero Gromov--Witten potential of $\mathcal X$ is 
\[
\mathcal F_{\mathcal X}^0({\bf t}):= \sum_{l, \beta}\frac{Q^{\beta}}{l!}\langle {\bf t,\ldots, t}\rangle_{0,l, \beta}^{\mathcal X},
\]
where $Q^\beta$ is an element of the Novikov ring which is a completion of the group ring $\mathbb C[\on{Eff}(\mathcal X)]$ of the semi group $\on{Eff}(X)$ of effective curve classes (\cite{Tseng}*{Section 2.5.2});
\[
{\bf t}=\sum_{i\geq 0}t_iz^i\in H^*_{\on{CR}}(\mathcal X)[z].
\].

 Givental's formalism for the genus-zero orbifold Gromov--Witten invariants in terms of a Lagrangian cone in Givental's symplectic vector space was developed in \cite{Tseng}.  Givental's symplectic vector space is
\[
\mathcal H:=H^*_{\on{CR}}(\mathcal X,\mathbb C)\otimes \mathbb C[\![\on{NE}(\mathcal X)]\!][z,z^{-1}]\!],
\]
where $\on{NE}(\mathcal X)$ is the Mori cone of $\mathcal X$. The symplectic form on $\mathcal H$ is defined as
\[
\Omega(f,g):=\on{Res}_{z=0}(f(-z),g(z))_{\on{CR}}dz,
\]
where $(-,-)_{\on{CR}}$ is the orbifold Poincar\'e pairing of the Chen--Ruan cohomology $H^*_{\on{CR}}(\mathcal X)$ of $\mathcal X$.

We consider the polarization
\[
\mathcal H=\mathcal H_+\oplus \mathcal H_-,
\]
\[
\mathcal H_+=H^*_{\on{CR}}(\mathcal X,\mathbb C)\otimes \mathbb C[\![\on{NE}(\mathcal X)]\!][z], \quad \mathcal H_-=z^{-1}H^*_{\on{CR}}(\mathcal X,\mathbb C)\otimes \mathbb C[\![\on{NE}(\mathcal X)]\!][\![z^{-1}]\!].
\]
 Givental's Lagrangian cone $\mathcal L_{\mathcal X}$ is defined as the graph of the differential of $\mathcal F^0_{\mathcal X}$ in the dilaton-shifted coordinates. That is,
\[
\mathcal L_{\mathcal X}:=\{(p,q)\in \mathcal H_-\oplus \mathcal H_+| p=d_q\mathcal F^0_{\mathcal X}\} \subset \mathcal H.
\]
The so-called $J$-function is a slice of $\mathcal L_{\mathcal X}$:
\[
J_{\mathcal X}(t,z):=z+t+\sum_{l, \beta}\sum_{\alpha}\frac{Q^{\beta}}{l!}\left\langle \frac{\phi_\alpha}{z-\bar{\psi}},t,\ldots,t\right\rangle^{\mathcal X}_{0,l+1,\beta}\phi^{\alpha},
\]
where 
\[
\{\phi_\alpha\}, \{\phi^\alpha\}\subset H^*_{\on{CR}}(\mathcal X)
\]
are additive bases dual to each other under orbifold Poincar\'e pairing and, 
\[
t=\sum_{\alpha}t^\alpha\phi_\alpha\in H^*_{\on{CR}}(\mathcal X).
\]
One can decompose the $J$-function according to the degree of curves
\[
J_{\mathcal X}(t,z)=\sum_{\beta}J_{\mathcal X, \beta}(t,z)Q^{\beta}.
\]

\section{A mirror theorem for multi-root stacks}

\subsection{A geometric construction of root stacks}\label{sec:construction}

Let $X$ be a smooth projective variety. Let $D:=D_1+D_2+...+D_n$ be a simple normal-crossing divisor with $D_i\subset X$ smooth and irreducible. Let $$\sigma_i:\mathcal{O}_X\to \mathcal{O}_X(D_i)$$ be a section such that $$\sigma_i^{-1}(0)=D_i.$$

We record the following easy property.
\begin{lemma}
 $$X_{D,\vec r}:=X_{(D_1, r_1), (D_2, r_2), ..., (D_n, r_n)}\simeq X_{(D_1, r_1)}\times_X X_{(D_2, r_2)}\times_X...\times_X X_{(D_n, r_n)}.$$
 \end{lemma}
\begin{proof}
To see this, it suffices to check that the groupoids of $S$-valued points are isomorphic for any scheme $S$.

The $S$-points of the left-hand side consist of $$f: S\to X, \{M_i: \text{ line bundle on }S\}, \{s_i\in H^0(M_i)\}, \{\phi_i:M_i^{\otimes r_i}\to f^*\mathcal{O}_X(D_i)\}$$
such that $s_i^{r_i}=\phi_i^*f^*\sigma_i$ for $i=1,...,n$.

The $S$-points of the $i$-th factor of the right-hand side consist of $$f: S\to X, M_i: \text{ line bundle on }S, s_i\in H^0(M_i), \phi_i:M_i^{\otimes r_i}\to f^*\mathcal{O}_X(D_i)$$
such that $s_i^{r_i}=\phi_i^*f^*\sigma_i$. 

The isomorphism is thus clear. 
\end{proof} 

The inertia stack of the multi-root stack can be described as follows. The coarse moduli spaces of twisted sectors are either $D_i$ or intersections of $D_i$. The isotropy groups are $\mu_{r_i}$ for $D_i$ and $\mu_{r_{i_1}}\times\cdots \times \mu_{r_{i_l}}$ for $\cap_{j=1}^l D_{i_j}$.

Next, we describe root stacks as complete intersections. The case of smooth divisors is treated in \cite{FTY}. 

We will construct the following tower
$$X_n\to ...\to X_2\to X_1\to X.$$
Here $\pi_{i+1}:X_{i+1}\to X_i$ is a $\mathbb{P}^1$-bundle, with $X_{i+1}:=\mathbb{P}(L_i\oplus \mathcal{O}_{X_i})$. Let $$X_{i0}:=\mathbb{P}(\mathcal{O}_{X_i}),\quad X_{i\infty}:=\mathbb{P}(L_i)\subset X_{i+1}.$$
Here $L_i:=\mathcal{O}_X(-D_i)$. Note that we have omitted various pull-backs from the notation.

There are natural isomorphisms $$\on{Hom}(L_i\oplus \mathcal{O}_{X_i}, \mathcal{O}_{X_i})\simeq H^0(X_i, (L_i\oplus \mathcal{O}_{X_i})^*)\simeq H^0(X_i, \pi_{i+1*}\mathcal{O}_{X_{i+1}}(1))\simeq H^0(X_{i+1}, \mathcal{O}_{X_{i+1}}(1)).$$
Let $$\tilde{f}_i\in H^0(X_{i+1}, \mathcal{O}_{X_{i+1}}(1))$$ be the image of $f_i:=(\sigma_i\oplus 1)^*\in \on{Hom}(L_i\oplus \mathcal{O}_{X_i}, \mathcal{O}_{X_i})$ under these isomorphisms.

Locally we have $X_{i0}=\{s_0=0\}$ and $X_{i\infty}=\{s_\infty=0\}$. We see that $$\tilde{f}_i^{-1}(0)=\{s_0\sigma_i(x)+s_\infty=0\}\simeq \sigma_i(X_i)\simeq X_i,$$
$$\tilde{f}_i^{-1}(0)\cap X_{i\infty}=\{s_0\sigma_i(x)+s_\infty=0, s_\infty=0\}\simeq \sigma_i^{-1}(0)=D_i \text{ pulled back to } X_i.$$
Set $$Y:=X_{n, (X_{1\infty}, r_1), (X_{2\infty}, r_2),...,(X_{n\infty}, r_n)}\overset{p_i}{\longrightarrow}X_{i+1}\to X.$$
Then we have
\begin{lemma}
$$\cap_{i=0}^{n-1}p_i^*\tilde{f}_i^{-1}(0)\simeq X_{(D_1, r_1), (D_2, r_2), ..., (D_n, r_n)}=X_{D,\vec r}.$$
\end{lemma}

\subsection{Mirror theorem}\label{sec:mir_thm}
Suppose that $r_1,...,r_n$ are pairwise co-prime, then $X_{D,\vec r}$ is nonsingular and its Gromov-Witten theory is well-defined.

Now assume $D_i$ are nef. Let
\[
h_i=c_1(\mathcal O_{X_i}(1)).
\]
Then, by the mirror theorem for toric fibrations \cite{Brown}, the $I$-function for $X_n$ is
\begin{align*}
    I_{X_n}(Q,q,t,z)=&e^{(\sum_{i=1}^n h_i\log q_i)/z}\sum_{\beta\in \on{NE}(X)}\sum_{d_i\geq 0, 1\leq i \leq n}J_{X,\beta}(t,z)Q^{\beta}q^d\\
& \prod_{i=1}^n\left(\frac{(\prod_{a\leq 0}(h_i+az))(\prod_{a\leq 0 }(h_i-D_i+az))}{(\prod_{a\leq d_i}(h_i+az))(\prod_{a\leq d_i-D_i\cdot \beta }(h_i-D_i+az))}\right).
\end{align*}
The $I$-function $I_{X_n}(Q,q,t,z)$ lies in Givental's Lagrangian cone $\mathcal L_{X_n}$ of $X_n$. Then $Y:=X_{n, (X_{1\infty}, r_1), (X_{2\infty}, r_2),...,(X_{n\infty}, r_n)}$, a multi-root stack of $X_n$, is a toric stack bundle over $X$. By the mirror theorem for toric stack bundles \cite{JTY}, the $I$-function is
\begin{align*}
    &I_{Y}(Q,q,t,z)=e^{(\sum_{i=1}^n h_i\log q_i)/z}\sum_{\beta\in \on{NE}(X)}\sum_{d_i\geq 0, 1\leq i \leq n}J_{X,\beta}(t,z)Q^{\beta}q^d\\
&\quad \cdot\prod_{i=1}^n\left(\frac{(\prod_{\langle a \rangle =\langle d_i/r_i\rangle, a\leq 0}(h_i/r_i+az))(\prod_{a\leq 0 }(h_i-D_i+az))}{(\prod_{\langle a \rangle =\langle d_i/r_i \rangle, a\leq d_i/r_i}(h_i/r_i+az))(\prod_{a\leq d_i-D_i\cdot \beta }(h_i-D_i+az))}\right) \textbf{1}_{-\left\langle \frac{d_1}{r_1},\frac{d_2}{r_2},\ldots, \frac{d_n}{r_n}\right\rangle}.
\end{align*}
The $I$-function $I_{Y}(Q,q,t,z)$ lies in Givental's Lagrangian cone $\mathcal L_{Y}$ of $Y$ by \cite{JTY}. 

Since $\mathcal O_{X_i}(1)$ are convex line bundles, we may apply orbifold quantum Lefschetz (\cite{Tseng}, \cite{CCIT09}). Recall that 
\[
\iota^*h_i=D_i,
\]
where $\iota: X_{D,\vec r}\hookrightarrow Y$ is the embedding. Following the construction in \cite{FTY}*{Section 3.2}, 
the non-extended $I$-function for the root stack $X_{D,\vec r}$ is
\begin{align}\label{I-orb}
&I_{X_{D,\vec r}}(Q,t,z)
=\sum_{\beta\in \on{NE}(X)} J_{X, \beta}(t,z)Q^{\beta}
\prod_{i=1}^n\frac{\prod_{0< a\leq d_i}(D_i+az)}{\prod_{\substack{\langle a \rangle =\langle d_i/r_i\rangle \\0< a\leq \frac{d_i}{r_i}}}(D_i/r_i+az)}\textbf{1}_{-\left\langle \frac{d_1}{r_1},\frac{d_2}{r_2},\ldots, \frac{d_n}{r_n}\right\rangle},
\end{align}
where $d_i=D_i\cdot \beta$.

We arrive at
\begin{thm}\label{thm-mirror}
Let $X$ be a smooth projective variety and $D:=D_1+D_2+...+D_n$ be a simple normal-crossing divisor with $D_i\subset X$ smooth, irreducible and nef. The non-extended $I$-function $I_{X_{D,\vec r}}$ (\ref{I-orb}) lies in Givental's Lagrangian cone $\mathcal L_{X_{D,\vec r}}$ of $X_{D,\vec r}$.
\end{thm}

\begin{rmk}
Similar to the case of smooth divisors, the nefness condition on $D_i$ can be removed if $D_i$ is a toric invariant divisor of a toric variety $X$.
\end{rmk}

\begin{ex}[Projective spaces]
Let $X=\mathbb{P}^{n-1}$. The small $J$-function of $X$ is well-known:
\begin{equation*}
J_{\mathbb{P}^{n-1}}=ze^{(t_0+Pt)/z}\sum_{d\geq 0}\frac{Q^d e^{d t}}{\prod_{0<a\leq d} (P+az)^n},    
\end{equation*}
where $P$ is the hyperplane class. 

If we take $D_1,...,D_n$ to be the toric prime divisors of $X$ (whose classes are equal to $P$), then the $I$-function for this $X_{D,\vec r}$ is
\begin{equation*}
ze^{(t_0+Pt)/z}\sum_{d\geq 0}\frac{Q^d e^{d t}}{\prod_{0<a\leq d} (P+az)^n}
\prod_{i=1}^n\frac{\prod_{0< a\leq d}(P+az)}{\prod_{\substack{\langle a \rangle =\langle d/r_i\rangle \\0< a\leq \frac{d}{r_i}}}(P/r_i+az)}\textbf{1}_{-\left\langle \frac{d}{r_1},\frac{d}{r_2},\ldots, \frac{d}{r_n}\right\rangle}.
\end{equation*}
After cancellation, this becomes 
\begin{equation*}
ze^{(t_0+Pt)/z}\sum_{d\geq 0}\frac{Q^d e^{d t}}{\prod_{i=1}^n\prod_{\substack{\langle a \rangle =\langle d/r_i\rangle \\0< a\leq \frac{d}{r_i}}}(P/r_i+az)}\textbf{1}_{-\left\langle \frac{d}{r_1},\frac{d}{r_2},\ldots, \frac{d}{r_n}\right\rangle}.
\end{equation*}
\end{ex}

\begin{ex}\label{ex-I-p-2-orb}
Let $X=\mathbb P^2$ and $D$ be the union of a line and a conic. The $I$-function for $\mathbb P^2_{D,\vec r}$ is
\begin{equation*}
ze^{(t_0+Pt)/z}\sum_{d\geq 0}\frac{Q^d e^{d t}}{\prod_{0<a\leq d} (P+az)^3}
\frac{\prod_{0< a\leq d}(P+az)}{\prod_{\substack{\langle a \rangle =\langle d/r_i\rangle \\0< a\leq \frac{d}{r_i}}}(P/r_i+az)}\frac{\prod_{0< a\leq 2d}(2P+az)}{\prod_{\substack{\langle a \rangle =\langle 2d/r_i\rangle \\0< a\leq \frac{2d}{r_i}}}(2P/r_i+az)}\textbf{1}_{-\left\langle \frac{d}{r_1},\frac{2d}{r_2}\right\rangle}.
\end{equation*}
\end{ex}

Similar to the case of smooth divisors, we can also write down the extended $I$-function for root stacks. The extended $I$-function depends on the choice of the extended data. For toric stack bundles (or toric stacks), extended data corresponds to extended stacky fans \cite{Jiang}. For example, we can choose the extended data to be a subset of the so-called box elements of the toric stack bundles. Box elements correspond to twisted sectors of the inertia stack of the toric stack bundle. Under suitable assumptions, it means that the extended $I$-function can be used to compute orbifold invariants with orbifold markings that map to these twisted sectors. An example of such an extended $I$-function is written in \cite{FTY}*{Section 3.4} for root stacks where the root construction is along a smooth divisor. 

For multi-root stacks, we can also choose the extended data to correspond to twisted sectors whose coarse moduli spaces are $D_i$. More precisely, we choose the extended data to be
\[
S=\{a_{ij}\}_{i\in\{1,\ldots,n\}, j\in \{1,\ldots,m\}}.
\]
This extended data corresponds to the box elements $a_{ij}/r_i$ of the toric stack bundle $Y$ (constructed in Section \ref{sec:construction}). 
Then the $S$-extended $I$-function is
\begin{align}\label{Extended-I-function}
&I_{X_{D,\vec r}}^{S}(Q,x,t,z)
=\sum_{\beta\in \on{NE}(X)}\sum_{\substack{(k_{i1},\ldots,k_{im})\in (\mathbb Z_{\geq 0})^m \\ 1\leq i \leq n}}J_{X, \beta}(t,z)Q^{\beta}\frac{\prod_{i=1}^n\prod_{j=1}^m x_{ij}^{k_{ij}}}{z^{\sum_{i=1}^n\sum_{j=1}^m k_{ij}}\prod_{i=1}^n\prod_{j=1}^m(k_{ij}!)}\times\\
\notag& \prod_{i=1}^n\left(
  \frac{\prod_{0<a\leq d_i}(D_i+az)\prod_{\langle a\rangle=\langle \frac{d_i-\sum_{j=1}^mk_{ij}a_{ij}}{r_i}\rangle,a\leq 0}(D_i/r_i+az)}{\prod_{\langle a\rangle=\langle \frac{d_i-\sum_{j=1}^mk_{ij}a_{ij}}{r_i}\rangle,a\leq \frac{d_i-\sum_{j=1}^mk_{ij}a_{ij}}{r_i}}(D_i/r_i+az)}\right)\mathbf 1_{\left\langle \frac{-d_1+\sum_{j=1}^mk_{1j}a_{1j}}{r_1},\ldots, \frac{-d_n+\sum_{j=1}^mk_{nj}a_{nj}}{r_n}\right\rangle},
\end{align}
where $x=\{x_{11},\ldots,x_{nm}\}$ is the set of variables corresponding to the extended data $S$. The variable $x_{ij}$ corresponds to the twisted sector whose coarse moduli space is $D_i$ and the age is $a_{ij}/r_i$.

\begin{thm}\label{thm:mirror-extended}
Let $X$ be a smooth projective variety and $D:=D_1+D_2+...+D_n$ be a simple normal-crossing divisor with $D_i\subset X$ smooth, irreducible and nef. The $S$-extended $I$-function $I_{X_{D,\vec r}}^{S}(Q,x,t,z)$ lies in Givental's Lagrangian cone $\mathcal L_{X_{D,\vec r}}$ of $X_{D,\vec r}$.
\end{thm}

\begin{rmk}
We can also choose the extended data to correspond to twisted sectors whose coarse moduli spaces are intersections of some $D_i$. There are many choices, we do not plan to write down more general version of the extended $I$-function.
\end{rmk}

\section{Large $\vec r$ limit}\label{sec:limit}

We will formally consider the large $\vec r$ limits (if they exist) of the invariants of root stacks as invariants of infinite root stacks. Following the description in \cite{FWY}*{Section 7.1}, we formally define the state space for the Gromov--Witten theory of $X_{D,\infty}$ as the limit of the state space of $X_{D,\vec r}$:
\[
\mathfrak H:=\bigoplus_{(s_1,\ldots,s_n)\in \mathbb Z^n}\mathfrak H_{(s_1,\ldots,s_n)},
\]
where 
\[
\mathfrak H_{(0,\ldots, 0)}:=H^*(X)
\]
\[
\mathfrak H_{(s_1,\ldots,s_n)}:=H^*(\cap_{i: s_i\neq 0} D_i).
\]
We write $[\gamma]_{(s_1,\ldots,s_n)}$ for an element in $\mathfrak H_{(s_1,\ldots,s_n)}$. Note that if $\cap_{i: s_i\neq 0} D_i=\emptyset$, then $\mathfrak H_{(s_1,\ldots,s_n)}=0$.

Similar to \cite{FTY}*{Section 4}, we fix a class $\beta \in \on{NE}(X)$ and let $r_i>d_i>0$. We see that the coefficient of the $I$-function is
\[
 J_{X, \beta}(t,z)Q^{\beta}
\prod_{i=1}^n\frac{\prod_{0< a\leq d_i}(D_i+az)}{(D_i+d_iz)/r_i}\textbf{1}_{-\left\langle \frac{d_1}{r_1},\frac{d_2}{r_2},\ldots, \frac{d_n}{r_n}\right\rangle}.
\]
The only factor that depends on $r_i$ is 
\[
\left(\prod_{i=1}^n r_i\right)\textbf{1}_{-\left\langle \frac{d_1}{r_1},\frac{d_2}{r_2},\ldots, \frac{d_n}{r_n}\right\rangle}.
\]
Note that when $d_i=0$, for some $i$, the corresponding hypergeometric factor is simply
\[
\frac{\prod_{0< a\leq d_i}(D_i+az)}{\prod_{\substack{\langle a \rangle =\langle d_i/r_i\rangle \\0< a\leq \frac{d_i}{r_i}}}(-D_i/r_i-az)}=1.
\]

We define the index set
\[
I_\beta=\left\{i\in\{1,\ldots,n\}| D_i\cdot \beta \neq 0\right\},
\]
then the marking with insertion $\textbf{1}_{\left\langle \frac{d_1}{r_1},\frac{d_2}{r_2},\ldots, \frac{d_n}{r_n}\right\rangle}$ maps to the intersection $\cap_{i\in I_\beta} D_i$. When the intersection is empty and $\textbf{1}_{\left\langle \frac{d_1}{r_1},\frac{d_2}{r_2},\ldots, \frac{d_n}{r_n}\right\rangle}=0$. For the corresponding coefficient of the $J$-function, such marking does not exist and invariants are simply zero.

Taking the limit of $r_i\rightarrow \infty$ for each $i$ to the non-extended $I$-function (\ref{I-orb}), then the non-extended $I$-function becomes
\begin{align}\label{I-snc}
I_{X_{D,\infty}}(Q,t,z):=\sum_{\beta\in \on{NE}(X)} J_{X, \beta}(t,z)Q^{\beta}
\prod_{i=1}^n\prod_{0< a< d_i}(D_i+az)[\textbf{1}]_{ (-d_1,-d_2,\ldots, -d_n)},
\end{align}
where, similar to \cite{FTY}*{Section 4.2}, we identify $$\left(\prod_{i: d_i>0} r_i\right)\textbf{1}_{-\left\langle \frac{d_1}{r_1},\frac{d_2}{r_2},\ldots, \frac{d_n}{r_n}\right\rangle}$$ with $$[\textbf{1}]_{ (-d_1,-d_2,\ldots, -d_n)}.$$
Recall that when $d_i=0$, the convention is that 
\[
\prod_{0< a< d_i}(D_i+az)=1.
\]
The existence of such a limit of the $I$-function implies that some genus zero invariants of root stacks stabilize as $r_i$ becomes sufficiently large. 

Similar to \cite{FTY}*{Section 4.3}, we can also write down a limit of the extended $I$-function (\ref{Extended-I-function}). The extended $I$-function for $X_{D,\infty}$ is denoted by $I^S_{X_{D,\infty}}(Q,x,t,z)$ or simply $I_{X_{D,\infty}}(Q,x,t,z)$ as the extended data is indicated by the variables $x_i$. For our propose, we only write down the part of the extended $I$-function of $X_{D,\infty}$ that takes value in $\mathfrak H_{(0,\ldots, 0)}:=H^*(X)$. This part of the $I$-function is denoted by $I_{X_{D,\infty},0}(Q,x,t,z)$. For simplicity, we set $a_{ij}=j$, for $1\leq i \leq n$ and $1\leq j\leq m$. Then,
\begin{align*}
I_{X_{D,\infty},0}(Q,x,t,z):=\sum_{\substack{\beta\in \on{NE}(X),(k_{i1},\ldots,k_{im})\in (\mathbb Z_{\geq 0})^m\\ \sum_{j=1}^m jk_{ij}= d_i, 1\leq i \leq n}  } & J_{X, \beta}(t,z)Q^{\beta}\frac{\prod_{i=1}^n\prod_{j=1}^m x_{ij}^{k_{ij}}}{z^{\sum_{i=1}^n\sum_{j=1}^m k_{ij}}\prod_{i=1}^n\prod_{j=1}^m(k_{ij}!)}\\
& \cdot\left(\prod_{i=1}^n\prod_{0<a\leq d_i}(D_i+az)\right).
\end{align*}

Using Birkhoff factorization procedure of \cite{CG}*{Corollary 5}, we know that invariants in the $J$-function of the root stack $X_{D,\vec r}$ stabilize as $\vec r\rightarrow \infty$ (by $\vec r\rightarrow \infty$, we mean $r_i\rightarrow \infty$ for each $i\in\{1,\ldots, n\}$). When invariants of $X_{D,\vec r}$ stabilize, we formally consider these invariants as invariants of $X_{D,\infty}$. The limits will be denoted by
\[
\langle \cdots\rangle^{X_{D,\infty}}.
\]
Since we consider invariants of $X_{D,\infty}$ as virtual counts of curves with tangency conditions along $D$, we will use the term ``orbifold marking'' and ``relative marking'' interchangeably. We will study the large $\vec r$ limit of invariants of root stacks (i.e., Conjecture \ref{conj-large-r}) in detail in a forthcoming paper.

Our $I$-function can be used to compute some genus zero invariants of $X_{D,\infty}$. We explain how the computation can be done through some examples where the numbers that we compute agree with the numbers of curves with tangency conditions along simple normal crossing divisors.
\begin{ex}\label{ex-p-2-orb-inv}
Let $X=\mathbb P^2$ and $D$ be the union of a line and a conic, which we considered in Example \ref{ex-I-p-2-orb}. We have 
\begin{align}
&I_{\mathbb P^2_{D,\infty},0}(Q,x,t_0,t,z):=\\
\notag&ze^{(t_0+Pt)/z}\sum_{\substack{d\geq 0,(k_{i1},\ldots,k_{im})\in (\mathbb Z_{\geq 0})^m\\ \sum_{j=1}^m jk_{ij}= d_i, 1\leq i \leq 2}  }  Q^d e^{d t}\frac{\prod_{i=1}^2\prod_{j=1}^m x_{ij}^{k_{ij}}}{z^{\sum_{i=1}^2\sum_{j=1}^m k_{ij}}\prod_{i=1}^2\prod_{j=1}^m(k_{ij}!)}\frac{\prod_{a=1}^{2d}(2P+az)}{\prod_{a=1}^d (P+az)^2}.
\end{align}
We can formally take $m$ to infinity (or for each $d$, we choose $m\geq 2d$ and compute invariants of degree $d$) and set $z=1$. To compute genus zero orbifold invariants with one orbifold marking for each irreducible component of $D$ and one interior marking with point constraint, we need to extract the following coefficient of $H^0(X)$ in $I_{\mathbb P^2_{D,\infty},0}$:
\[
\sum_{d\geq 0}Q^dx_{1d}x_{2(2d)}\frac{(2d)!}{(d!)^2}.
\]
Recall that the mirror map is the coefficient of the $z^0$-coefficient of the $I$-function. By direct computation, the mirror map is trivial. Note that the variable $x_{ij}$ corresponds to relative marking that has contact order $j$ to $D_i$ and the non-negative integer $k_{ij}$ records how many times such relative marking appears. In here, $k_{1(d)}=k_{2(2d)}=1$ and $k_{ij}=0$ for all other $i,j$. Hence, we have
\[
\left\langle [\mathbf 1]_{(d,0)},[\mathbf 1]_{(0,2d)},[pt]\right\rangle_{0,(d),(2d),1,d}^{\mathbb P^2_{D,\infty}}=\frac{(2d)!}{(d!)^2}.
\]
This agrees with the expectation that the number of degree $d$ curves in $\mathbb{P}^2$ through one generic point, meeting with a line and a conic with maximal contact orders is $\frac{(2d)!}{(d!)^2}$. The corresponding log Gromov--Witten invariants are computed by Bousseau--Brini--van Garrel in a forthcoming paper \cite{BBvG20} using the scattering diagram. 
\end{ex}

\begin{ex}\label{ex:p1-p1}
Let $X=\mathbb P^1\times \mathbb P^1$ and $D=L_1+L_2$, where $L_1$ and $L_2$ are distinct $(1,1)$ curves. Following the same process in Example \ref{ex-p-2-orb-inv}, we can compute genus zero orbifold invariants of $\mathbb P^1\times \mathbb P^1$ with one orbifold marking for each $L_1$ and $L_2$ and one interior marking with point constraint:
\[
\left\langle [\mathbf 1]_{(d_1,d_2)},[\mathbf 1]_{(d_1,d_2)},[pt]\right\rangle_{0,\{(d_1),(d_2)\},\{(d_1),(d_2)\},1,(d_1,d_2)}^{(\mathbb P^1\times \mathbb P^1)_{D,\infty}}=\frac{(d_1+d_2)!^2}{(d_1!)^2(d_2!)^2}.
\]
They coincide with the computation of log Gromov--Witten invariants in \cite{NR}.
\end{ex}

\section{The local-log-orbifold principle calculations}\label{sec:local-log-orb}

\subsection{Smooth divisors}\label{sec:sm_div}
We consider the case when the divisor $D$ is smooth.
\subsubsection{Computation using the non-extended $I$-function}

For smooth divisors, the non-extended $I$-function for the pair $(X,D)$ is given in \cite{FTY}*{Theorem 1.4}:

\begin{align*}
I_{(X,D)}(Q,t,z):=\sum_{\beta\in \on{NE}(X) }J_{X, \beta}(t,z)Q^{\beta}
\left(\prod_{0<a\leq d-1}(D+az)\right)[\mathbf 1]_{-d}.
\end{align*}
The local $I$-function is
\begin{equation}
I_{\mathcal O_X(-D)}(Q,t,z)=\sum_{\beta\in \on{NE}(X)} J_{X, \beta}(t,z)Q^{\beta}
\prod_{0\leq a< d}(-D+\lambda-az),
\end{equation}
where $\lambda$ is the equivariant parameter.

\begin{proof}[Proof of Theorem \ref{thm-local-rel}]

Note that the coefficient of $[\delta]_{-d}$ for $I_{(X,D)}$ is exactly the same as the coefficient of $D\cup \delta$ for $I_{\mathcal O_X(-D)}$. The way to identify local and relative $I$-functions is to identify $[\delta]_{-d}$ with $D\cup \delta\in H^*(X)$. 
More precisely, let $\iota: D \hookrightarrow X$ be the inclusion map, we have
\[
\iota_{!}I_{(X,D),\beta}(Q,t,z)=(-1)^{d-1}\left[I_{\mathcal O_X(-D),\beta}(Q,t,z)\right]_{\lambda=0}.
\]

When $t=0$ and the mirror maps are trivial, we directly have the following relation between relative and local invariants:
\[
\left\langle [\iota^*\gamma]_{d}\bar{\psi}^a\right\rangle_{0,0,(d),\beta}^{(X,D)}=(-1)^{d-1}\left\langle D\cdot \gamma \bar{\psi}^a \right\rangle_{0,1,\beta}^{\mathcal O_{X}(-D)},  
\]
where $\gamma \in H^*(X)$; the left-hand side is the genus zero relative Gromov--Witten invariant of $(X,D)$ with one marking which has to be a relative marking with maximal contact order; the right-hand side is the genus zero local invariant of $\mathcal O_{X}(-D)$.
This is slightly more general than the result in \cite{vGR}. It can also be proved following \cite{vGR} with minor adjustment. 

When mirror maps are not trivial, we can compute the inverse mirror maps in both cases. The inverse mirror maps can be identified under the identification: $\iota_{!}[\delta]_{-i}=D\cup \delta$, for $i\in \mathbb Z_{>0}$. Therefore, relative and local invariants coincide. More generally, the Birkhoff factorization procedure of \cite{CG}*{Corollary 5} can be used to compute the invariants in the $J$-functions. Therefore, we have the following equality:
\[
\left\langle \prod_{i=1}^l [\gamma_i]_0, [\iota^*\gamma]_{d}\bar{\psi}^a\right\rangle_{0,l,(d),\beta}^{(X,D)}=(-1)^{d-1}\left\langle \prod_{i=1}^l\gamma_i, D\cdot \gamma \bar{\psi}^a \right\rangle_{0,l+1,\beta}^{\mathcal O_{X}(-D)},  
\]
where $\gamma_i\in H^*(X)$ and $a\in \mathbb Z_{\geq 0}$ such that the virtual dimension constraint holds.
\end{proof}

Note that when $\gamma=\textbf{1}\in H^0(X)$ and $a=0$, this becomes the original version of the log-local principal in \cite{vGR} by the divisor equation.

\begin{rmk}
Note that the descendant classes that we use in relative Gromov--Witten theory are descendant classes pulled back from the moduli space $\bM_{g,l+1}(X,\beta)$ of stable maps to $X$.
\end{rmk}

\subsubsection{Computation using the extended $I$-function}\label{sec:extend-rel-I}
The extended $I$-function for relative invariants is given in \cite{FTY}*{Section 4.3}. The $I$-function is taken as a limit of the $I$-function for root stacks. Here, we consider the component of the relative $I$-function that takes values in $H^*(X)$:

\begin{equation}\label{eqn:I0}
I_{(X,D),0}(Q,x,t,z):=\sum_{\substack{\beta\in \on{NE}(X),(k_1,\ldots,k_m)\in (\mathbb Z_{\geq 0})^m\\ \sum_{i=1}^m ik_i= d} }J_{X, \beta}(t,z)Q^{\beta}\frac{\prod_{i=1}^m x_i^{k_i}}{z^{\sum_{i=1}^m k_i}\prod_{i=1}^m(k_i!)}
\left(\prod_{0<a\leq d}(D+az)\right).
\end{equation}

Furthermore, for each $\beta$, we can take $m\geq d=D\cdot \beta$. For relative $I$-function, we only consider the summand where $k_{d}=1$ and $k_i=0$ for $i\neq d$. In other words, we take the degree one part of the polynomial in $x_1,\ldots, x_m$, which is simply $x_d$ because we need $\sum_{i=1}^m ik_i= d$. This corresponds to relative invariants with one relative marking of insertion $[\textbf{1}]_{d}$. The corresponding coefficient of the relative $I$-function is
\begin{align}\label{rel-I-ext-coeff}
J_{X, \beta}(t,z) \frac{x_{d}}{z}Q^{\beta}\left(\prod_{0<a\leq d}(D+az)\right).
\end{align}
When mirror maps are trivial, (\ref{rel-I-ext-coeff}) equals the part of the relative $J$-function where invariants are with one relative marking (with maximal tangency) of insertion $[\textbf{1}]_{d}$ and the distinguished marking is an interior marking.

The local $I$-function is
\begin{equation}
I_{\mathcal O_X(-D)}(Q,t,z)=\sum_{\beta\in \on{NE}(X)} J_{X, \beta}(t,z)Q^{\beta}
\prod_{0\leq a< d}(-D+\lambda-az).
\end{equation}
We will see that the local $I$-function and the extended relative $I$-function are related by derivatives. We consider the following operator called the $S$-operator
\begin{equation}\label{S-operator}
S_{\mathcal O_X(-D)}(t,z)(\gamma):= \gamma+\sum_{l, \beta}\sum_\alpha \frac{Q^\beta}{l!}\left\langle t,...,t,\gamma, \frac{\phi_\alpha}{z-\bar{\psi}}\right\rangle^{\mathcal O_X(-D)}_{0,l+2,\beta}\phi^\alpha,
\end{equation}
where $\gamma\in H^*(X;\mathbb Q)$.
It can be written as a derivative of the $J$-function. We are interested in the case when $\gamma=D=\sum_{j}c_j\phi_j\in H^2(X)$, for some constant $c_j$, where $\{\phi_j\}_j$ is a basis of $H^2(X)$. Then
\[
S_{\mathcal O_X(-D)}(t,z)(D)=\sum_{j}c_j zt_{j}\partial_{t_{j}}J_{\mathcal O(-D)}(t,z).
\]
Now consider the corresponding derivative of the local $I$-function
\begin{align*}
&\sum_{j}c_j zt_{j}\partial_{t_{j}}I_{\mathcal O(-D)}(Q,t,z)\\
=&\sum_{j}c_j zt_{j}\partial_{t_{j}}\sum_{\beta\in \on{NE}(X)} J_{X, \beta}(t,z)Q^{\beta}\prod_{0\leq a< d}(-D+\lambda-az)\\
=&\sum_{\beta\in \on{NE}(X)} J_{X, \beta}(t,z)Q^{\beta}(D+dz)\prod_{0\leq a< d}(-D+\lambda-az),
\end{align*}
where the third line follows from the definition of the $J$-function of $X$ and the divisor equation.

Therefore, we can identify the $I$-functions:
\begin{align}\label{equ:local-rel-I-extend}
zI_{(X,D),0,\beta,x_d}(Q,x,t,z)=(-1)^{d-1}\left[\frac{1}{D+\lambda}\sum_{j}c_j zt_{j}\partial_{t_{j}}I_{\mathcal O(-D),\beta}(Q,t,z)\right]_{\lambda=0},
\end{align}
where $I_{(X,D),0,\beta,x_d}(Q,x,t,z)$ is the $x_d$-coefficient of $I_{(X,D),0,\beta}(Q,x,t,z)$. Let $t=0$ and mirror maps be trivial, Equation (\ref{equ:local-rel-I-extend}) directly imply the relation between local invariants in $S_{\mathcal O_X(-D)}(0,z)(D)$ and the corresponding relative invariants in the relative $I$-function. More specifically, we have
\[
\left\langle [\mathbf 1]_d, [\gamma]_{0}\bar{\psi}^a\right\rangle_{0,1,(d),\beta}^{(X,D)}=(-1)^{d-1}\left\langle D,\gamma\bar{\psi}^a \right\rangle_{0,2,\beta}^{\mathcal O_{X}(-D)},  
\]
where the right-hand side can be simplified using the divisor equation.

Similar to the case with the non-extended $I$-function, following the Birkhoff factorization procedure of \cite{CG}*{Corollary 5}, we can allow more general targets and more ordinary markings. Therefore, we can recover the equality between local and log Gromov--Witten invariants in \cite{vGR}.

\subsection{Normal crossing divisors}\label{sec:snc_div}

\subsubsection{Computation using the non-extended $I$-function}

\begin{proof}[Proof of Theorem \ref{thm-local-orb}]
Recall that the local $I$-function of $\oplus_{i=1}^n\sO_X(-D_i)$ is
\[
I_{\oplus_{i=1}^n\sO_X(-D_i)}(Q,t,z):=\sum_{\beta\in \on{NE}(X)} J_{X, \beta}(t,z)Q^{\beta}
\prod_{i=1}^n\prod_{0\leq a< d_i}(-D_i+\lambda_i-az),
\]
where $\lambda_i$ are equivariant parameters.
Let $\beta$ be a curve class of $X$ with $d_i:=D_i\cdot \beta>0$ for $i\in \{1,\ldots, n\}$. The non-extended $I$-function (\ref{I-snc}) of $X_{D,\infty}$ can be identified with the local $I$-function similar to the smooth divisor case:
\[
\iota_{!}I_{X_{D,\infty},\beta}(Q,t,z)=\left(\prod_{i=1}^n(-1)^{d_i-1}\right)\left[I_{\oplus_{i=1}^n\sO_X(-D_i),\beta}(Q,t,z)\right]_{\lambda_i=0}.
\]
We assume that the intersection of the divisors $\cap_{i=1}^n D_i$ is not empty. The corresponding $J$-function has only one relative marking. The only relative marking has to map to the intersection $\cap_{i=1}^n D_i$ and has maximal tangency to all $D_i$.

Let $t=0$ and mirror maps be trivial, we directly have the following relation between orbifold and local invariants:
\[
\left\langle [\iota^*\gamma]_{\vec d}\bar{\psi}^a\right\rangle_{0,0,(\vec d),\beta}^{X_{D,\infty}}=\left(\prod_{i=1}^n(-1)^{d_i-1}\right)\left\langle (\cup_{i=1}^n D_i)\cdot \gamma \bar{\psi}^a \right\rangle_{0,1,\beta}^{\oplus_{i=1}^n\sO_X(-D_i)},  
\]
where $\gamma\in H^*(X)$.

Givental's formalism for the formal Gromov--Witten theory of infinite root stacks has been built recently in \cite{TY20c}. We can also simple proceed as follows using Givental's formalism for Gromov--Witten theory of $X_{D,\vec r}$. Given a curve class $\beta\in \on{NE}(X)$, we consider the Gromov--Witten theory of $X_{D,\vec r}$ for a sufficiently large $\vec r$. Applying the Birkhoff factorization procedure of \cite{CG}*{Corollary 5} to the Gromov--Witten theory of $X_{D,\vec r}$ and the Gromov--Witten theory of $\oplus_{i=1}^n\sO_X(-D_i)$, we have the following equality:
\[
\left\langle \prod_{i=1}^l[\gamma_i]_0, [\iota^*\gamma]_{\vec d}\bar{\psi}^a\right\rangle_{0,l,(\vec d),\beta}^{X_{D,\infty}}=\left(\prod_{i=1}^n(-1)^{d_i-1}\right)\left\langle \prod_{i=1}^l \gamma_i, (\cup_{i=1}^n D_i)\cdot \gamma \bar{\psi}^a \right\rangle_{0,l+1,\beta}^{\oplus_{i=1}^n\sO_X(-D_i)},  
\]
where $\gamma, \gamma_i\in H^*(X)$ for $i=1,\ldots, l$.
\end{proof}

\subsubsection{Computation using the extended $I$-function}

Recall that, the part of the extended $I$-function that takes value in $H^*(X)$ is
\begin{align*}
I_{X_{D,\infty},0}(Q,x,t,z):=\sum_{\substack{\beta\in \on{NE}(X),(k_{i1},\ldots,k_{im})\in (\mathbb Z_{\geq 0})^m\\ \sum_{j=1}^m jk_{ij}= d_i, 1\leq i \leq n}  } & J_{X, \beta}(t,z)Q^{\beta}\frac{\prod_{i=1}^n\prod_{j=1}^m x_{ij}^{k_{ij}}}{z^{\sum_{i=1}^n\sum_{j=1}^m k_{ij}}\prod_{i=1}^n\prod_{j=1}^m(k_{ij}!)}\\
& \cdot\left(\prod_{i=1}^n\prod_{0<a\leq d_i}(D_i+az)\right).
\end{align*}

\begin{proof}[Proof of Theorem \ref{thm-local-orb-ext}]

Let $\beta$ be a curve class of $X$ with $d_i:=D_i\cdot \beta>0$ for $i\in \{1,\ldots, n\}$. We can compute genus zero invariants with one orbifold/relative marking for each $D_i$ and one interior marking. Similar to the smooth case, the corresponding part of the $I$-function $I_{X_{D,\infty},0}$ that takes value in $H^*(X)$ is
\[
J_{X, \beta}(t,z)\left(\prod_{i=1}^n\frac{x_{id_i}}{z}\right)Q^{\beta}\left(\prod_{i=1}^n\prod_{0<a\leq d_i}(D_i+az)\right).
\] 

To compute the corresponding local invariants, we consider the $I$-function for the local Gromov-Witten theory of $\bigoplus_{i=1}^n \mathcal O_X(-D_i)$:
\[
I_{\oplus_{i=1}^n\sO_X(-D_i)}(Q,t,z):=\sum_{\beta\in \on{NE}(X)} J_{X, \beta}(t,z)Q^{\beta}
\prod_{i=1}^n\prod_{0\leq a< d_i}(-D_i+\lambda_i-az).
\]
Similar to the case of smooth divisors, $I_{\oplus_{i=1}^n\sO_X(-D_i)}(Q,t,z)$ and $I_{X_{D,\infty},0}$ are related by derivatives. Let $\{\phi_j\}_j$ is a basis of $H^2(X)$. We have $D_i=\sum_{j}c_{ij}\phi_j\in H^2(X)$, for some constant $c_{ij}$. Now consider the corresponding derivative of the local $I$-function
\begin{align*}
&\prod_{i=1}^n\left(\sum_{j}c_{ij} t_{j}\partial_{t_{j}}\right)I_{\oplus_{i=1}^n\sO_X(-D_i)}(Q,t,z)\\
=&\prod_{i=1}^n\left(\sum_{j}c_{ij} t_{j}\partial_{t_{j}}\right)\sum_{\beta\in \on{NE}(X)} J_{X, \beta}(t,z)Q^{\beta}\prod_{i=1}^n\prod_{0\leq a< d_i}(-D_i+\lambda_i-az)\\
=&\sum_{\beta\in \on{NE}(X)} J_{X, \beta}(t,z)Q^{\beta}\prod_{i=1}^n\frac{D_i+d_iz}{z}\prod_{i=1}^n\prod_{0\leq a< d_i}(-D_i+\lambda_i-az).
\end{align*}

Therefore, we can identify the $I$-functions:
\begin{align}
&I_{X_{D,\infty},0,\beta,\prod_{i=1}^n x_{id_i}}(Q,x,t,z)\\
\notag=&\left(\prod_{i=1}^n(-1)^{d_i-1}\right)\left[\frac{1}{\prod_{i=1}^n (D_i+\lambda_i)}\prod_{i=1}^n\left(\sum_{j}c_{ij} t_{j}\partial_{t_{j}}\right)I_{\oplus_{i=1}^n\sO_X(-D_i),\beta}(Q,t,z)\right]_{\lambda_i=0},
\end{align}
where $I_{X_{D,\infty},0,\beta,\prod_{i=1}^n x_{id_i}}(Q,x,t,z)$ is the $\prod_{i=1}^n x_{id_i}$-coefficient of $I_{X_{D,\infty},0,\beta}(Q,x,t,z)$; $\lambda_i$ are equivariant parameters.

When mirror maps are trivial, we directly obtain the relation between orbifold invariants and local invariants: 
\[
\left\langle \prod_{i=1}^n[\mathbf 1]_{(0,\ldots,0,d_i,0,\ldots,0)}, [\gamma]_{0}\bar{\psi}^a\right\rangle_{0,1,(d_1),\ldots, (d_n),\beta}^{X_{D,\infty}}=\left(\prod_{i=1}^n (-1)^{d_i-1}\right)\left\langle \prod_{i=1}^n D_i, \gamma\bar{\psi}^a \right\rangle_{0,n+1,\beta}^{\bigoplus_{i=1}^n \mathcal O_X(-D_i)}.  
\]
In general, we can apply the Birkhoff factorization procedure of \cite{CG}*{Corollary 5} to the Gromov--Witten theory of $X_{D,\vec r}$ and the Gromov--Witten theory of $\oplus_{i=1}^n\sO_X(-D_i)$ to obtain the following equality:
\begin{align*}
    &\left\langle\prod_{i=1}^l [\gamma_i]_0, \prod_{i=1}^n [\textbf 1]_{(0,\ldots,0,d_i,0,\ldots,0)}, [\gamma]_{0}\bar{\psi}^a\right\rangle_{0,l+1,(d_1),\ldots, (d_n),\beta}^{X_{D,\infty}}\\
    =&\left(\prod_{i=1}^n (-1)^{d_i-1}\right)\left\langle \prod_{i=1}^l \gamma_i, \prod_{i=1}^nD_i,\gamma\bar{\psi}^a \right\rangle_{0,n+l+1,\beta}^{\bigoplus_{i=1}^n \mathcal O_X(-D_i)},  
\end{align*}
where $\gamma,\gamma_i\in H^*(X)$ for $i=1,\ldots, l$.
\end{proof}

\begin{ex}
Let $X$ be $\mathbb P^2$ and $D$ be a union of a line and a conic. In Example \ref{ex-p-2-orb-inv}, we computed that 
\[
\left\langle [\mathbf 1]_{(d,0)},[\mathbf 1]_{(0,2d)},[pt]\right\rangle_{0,(d),(2d),1,d}^{\mathbb P^2_{D,\infty}}=\frac{(2d)!}{(d!)^2}.
\]
The corresponding local invariants of $\mathcal O_{\mathbb P^2}(-1)\bigoplus \mathcal O_{\mathbb P^2}(-2)$ are computed in \cite{KP}*{Proposition 2}. They can be computed using the local $I$-function of $\mathcal O_{\mathbb P^2}(-1)\bigoplus \mathcal O_{\mathbb P^2}(-2)$ as well:
\[
\left\langle [pt]_{\mathbb P^2}\right\rangle_{0,1,d}^{\mathcal O_{\mathbb P^2}(-1)\bigoplus \mathcal O_{\mathbb P^2}(-2)}=(-1)^{d}\frac{(2d)!}{2d^2(d!)^2}.
\]
Therefore, 
\[
\left\langle [\mathbf 1]_{(d,0)},[\mathbf 1]_{(0,2d)},[pt]\right\rangle_{0,(d),(2d),1,d}^{\mathbb P^2_{D,\infty}}=(-1)^d 2d^2\left\langle [pt]_{\mathbb P^2}\right\rangle_{0,1,d}^{\mathcal O_{\mathbb P^2}(-1)\bigoplus \mathcal O_{\mathbb P^2}(-2)}.
\]
\end{ex}

\begin{ex}
Let $X=\mathbb P^1\times \mathbb P^1$ and $D=L_1+L_2$, where $L_1$ and $L_2$ are distinct $(1,1)$ curves. In Example \ref{ex:p1-p1}, we computed that
\[
\left\langle [\mathbf 1]_{(d_1,d_2)},[\mathbf 1]_{(d_1,d_2)},[pt]\right\rangle_{0,\{(d_1),(d_2)\},\{(d_1),(d_2)\},1,(d_1,d_2)}^{(\mathbb P^1\times \mathbb P^1)_{D,\infty}}=\frac{(d_1+d_2)!^2}{(d_1!)^2(d_2!)^2}.
\]
The corresponding local invariants of $\mathcal O_{\mathbb P^1\times \mathbb P^1}(-1,-1)\bigoplus \mathcal O_{\mathbb P^1\times \mathbb P^1}(-1,-1)$ are computed in \cite{KP}*{Proposition 3}. They can also be computed using the local $I$-function. We have
\[
\left\langle [pt]\right\rangle_{0,1,(d_1,d_2)}^{\mathcal O_{\mathbb P^1\times \mathbb P^1}(-1,-1)\bigoplus \mathcal O_{\mathbb P^1\times \mathbb P^1}(-1,-1)}=\frac{1}{(d_1+d_2)^2}\frac{(d_1+d_2)!^2}{(d_1!)^2(d_2!)^2}.
\]
Therefore, 
\begin{align*}
    &\left\langle [\mathbf 1]_{(d_1,d_2)},[\mathbf 1]_{(d_1,d_2)},[pt]\right\rangle_{0,\{(d_1),(d_2)\},\{(d_1),(d_2)\},1,(d_1,d_2)}^{(\mathbb P^1\times \mathbb P^1)_{D,\infty}}\\
    =&(d_1+d_2)^2\left\langle [pt]\right\rangle_{0,1,(d_1,d_2)}^{\mathcal O_{\mathbb P^1\times \mathbb P^1}(-1,-1)\bigoplus \mathcal O_{\mathbb P^1\times \mathbb P^1}(-1,-1)}.
\end{align*}

\end{ex}
\section{Quantum periods for Fano varieties and classical periods}

Let $X$ be a Fano variety and let $D\in |-K_X|$ be a reduced simple normal crossing divisor. It is expected that the classical period associated to the mirror Landau--Ginzburg potential W of $(X,D)$ is related to the regularized quantum period of the Fano variety $X$ defined in the Fano search program. The classical periods of $W$ can be defined using log Gromov--Witten invariants following the Frobenius structure conjecture of \cite{GHK}. A special case of such result has been proved in \cite{Mandel}. In this section, instead of considering log Gromov--Witten invariants of $(X,D)$, we study the corresponding orbifold Gromov--Witten invariants of $X_{D,\infty}$ and their relation with regularized quantum periods.

\subsection{Quantum periods}

Given a Fano variety, one can define the quantum period. 

\begin{defn}[\cite{CCGGK}*{Definition 4.2}]
The quantum period of a Fano variety $X$ is the power series
\[
G_X(t):=\sum_{m\geq 0}p_mt^m,
\]
where
\[
p_0:=1, \quad p_1:=0,
\]
and
\[
p_m:=\sum_{\beta \in \on{NE}(X), -K_X\cdot \beta=m}\int_{[\bM_{0,1}(X,\beta)]^{\on{vir}}}\psi^{m-2}\on{ev}^*([pt]),
\]
for $m\geq 2$.
\end{defn}

\begin{defn}[\cite{CCGGK}*{Definition 4.8}]
The regularized quantum period is the Fourier--Laplace transform of the quantum period:
\[
\hat{G}_X(t):=\sum_{m\geq 0}(m!)p_mt^m.
\]
\end{defn}

\subsection{The Frobenius structure conjecture}\label{sec:frobenius}

Now we turn to the mirror side. A mirror Landau--Ginzburg potential $W$ of $X$ can be associated with a classical period $\pi_W$, which is defined in terms of the constant terms of powers of $W$. A Landau--Ginzburg potential $W$ is said to be a mirror of $X$ if the regularized quantum period of $X$ coincides with the classical period of $W$:
\[
\hat{G}_X(t)=\pi_W.
\]
Given a Fano variety $X$, let $D=D_1+\cdots+D_n\in |-K_X|$. The Frobenius structure conjecture of \cite{GHK} provides a precise way to construct the superpotential $W$. Note that the conjecture was stated in Section 0.4 of the first arXiv version of \cite{GHK}. We briefly review this conjecture and the definition of classical period of $W$ in this section.

It is well known that the degree $0$ subalgebra $QH_{\log}^0(X)$ of the quantum cohomology ring $QH_{\log}^*(X)$ admits quantum product structure. By \cite{GS}, the degree $0$ subalgebra $QH_{\log}^0(X,D)$ of the relative quantum cohomology ring $QH_{\log}^*(X,D)$ is equipped with the relative quantum product structure. By the Frobenius structure conjecture of \cite{GHK}, $QH_{\log}^0(X,D)$ should be naturally equipped with a canonical basis of ``theta functions''. 

Let $S$ be the dual intersection complex of $D$. That is, $S$ is the simplicial complex with vertices $v_1,\ldots, v_n$ and simplices $\langle v_{i_1},\ldots, v_{i_p}\rangle$ corresponding to non-empty intersections $D_{i_1}\cap \cdots \cap D_{i_p}$. Let $B$ denote the cone over $S$ and $\Sigma$ be the induced simplicial fan in $B$. Let $B(\mathbb Z)$ be the set of integer points of $B$. There is a bijection between points $p\in B(\mathbb Z)$ and prime fundamental classes $\vartheta_p\in QH_{\log}^0(X,D)$. 

Suppose we are given points $q_1,\ldots, q_s\in B_0(\mathbb Z)$, where $B_0=B\setminus \{0\}$. Each $q_i$ can be written as a linear combination of primitive generators $v_{ij}$ of rays in $\Sigma$:
\[
q_i=\sum_{j} m_{ij}v_{ij},
\]
where the ray generated by $v_{ij}$ corresponds to a divisor $D_{ij}$. 

For $s\geq 2$, using the result of \cite{GS13} and \cite{AC14}, one can define the associated log Gromov--Witten invariant
\[
N_\beta(q_1,\ldots, q_s):=\int_{[\bM_{0,s+1}(X/D,\beta)]^{\on{vir}}}\on{ev}_0^*[pt]\cdot \psi_0^{s-2},
\]
where $\bM_{0,s+1}(X/D,\beta)$ is the moduli stack of logarithmic stable maps which provides a compactification for the space of stable maps
\[
g:(C,p_0,p_1,\ldots,p_s)\rightarrow X
\]
such that $g_*[C]=\beta$ and $C$ meets $D_{ij}$ at $p_i$ with contact order $m_{ij}$ for each $i,j$ and contact order zero with $D$ at $p_0$.

One can define a $\mathbb Q [\on{NE}(X)]$-multilinear symmetric $s$-point function
\[
\langle \cdot \rangle: QH_{\log}^0(X,D)^s\rightarrow \mathbb Q[\on{NE}(X)]
\]
as the following: for $s\geq 2$ and $q_1,\ldots,q_s\in B_0(\mathbb Z)$,
\[
\left\langle \vartheta_{q_1}, \ldots, \vartheta_{q_s}\right\rangle:=\sum_{\beta\in \on{NE}(X)}N_\beta (q_1,\ldots,q_s)Q^\beta\in  \mathbb Q [\on{NE}(X)].
\]
For $s=1$, we have
\[
\langle \vartheta_0\rangle=1, \text{ and } \langle \vartheta_q\rangle=0, \text{ for } q\in B_0(\mathbb Z). 
\]
Finally,
\[
\left\langle\vartheta_0, \vartheta_{q_1}, \ldots, \vartheta_{q_s}\right\rangle:=\left\langle \vartheta_{q_1}, \ldots, \vartheta_{q_s}\right\rangle, \text{ for } s\geq 1.
\]

The Frobenius structure conjecture of \cite{GHK} can be stated as follows.

\begin{conj}[Frobenius structure conjecture]
There is a unique associative product $*$ on $QH^0(X,D)$ such that
\[
\left\langle \vartheta_{q_1}, \ldots, \vartheta_{q_s}\right\rangle=\left\langle \vartheta_{q_1} *\cdots * \vartheta_{q_s}\right\rangle.
\]
\end{conj}
The conjecture was proved in \cite{GS} for $s\leq 3$ by explicitly defining all structure constants in terms of punctured Gromov--Witten invariants. It was proved for cluster log pairs in \cite{Mandel19} and for affine log Calabi--Yau varieties containing a torus in \cite{KY}.

Now we turn to the classical period of a mirror Landau--Ginzburg superpotential $W$ of a Fano variety $X$. Let $D_1,\ldots,D_n$ be irreducible components of $D$. The superpotential $W$ is defined as
\[
W:=\vartheta_{[D_1]}+\cdots+\vartheta_{[D_n]},
\]
where $\vartheta_{[D_i]}:=\vartheta_{v_i}$ and $v_i$ is the primitive generator of the ray in $\Sigma$ corresponding to $D_i$.
The classical period of $W$ is defined as the sum of the $\vartheta_{0}$-coefficient, denoted by $c_{W,d,0}$, of $W^d$ for $d\in \mathbb Z_{\geq 0}$. In other words, the classical period of $W$ is
\[
\pi_W=\sum_{d=0}^\infty c_{W,d,0}\in \mathbb Q[\on{NE}(X)].
\]
More specifically, we have
\[
W^d:= \sum_{q\in B(\mathbb Z)}c_{W,d,q}\vartheta_q=\sum_{(d_1,\ldots, d_n)\in (\mathbb Z_{\geq 0})^n: \sum_{i=1}^n d_i=d} \frac{d!}{d_1!\cdots d_n!} \vartheta_{[D_1]}^{d_1}*\cdots *\vartheta_{[D_n]}^{d_n}.
\]
Therefore, the $\vartheta_{0}$-coefficient of $W^d$ is
\[
c_{W,d,0}=\sum_{(d_1,\ldots, d_n)\in (\mathbb Z_{\geq 0})^n: \sum_{i=1}^n d_i=d} \frac{d!}{d_1!\cdots d_n!} \sum_{\beta \in \on{NE}(X)}Q^\beta N_\beta(\textbf{q}_{(d_1,\ldots, d_n)}),
\]
where $\textbf{q}_{(d_1,\ldots, d_n)}$ is a $d$-tuple consisting of $d_i$ instances of $[D_i]$ for each $i=1,\ldots,n$.

It is expected that the regularized quantum period coincides with the classical period defined using $W$. This has been proved for some special cases in \cite{Mandel}.

\subsection{Computation using root stack invariants}

Log Gromov--Witten invariants of $(X,D)$ and formal Gromov--Witten invariants of $X_{D,\infty}$ both provide virtual counts of curves with tangency conditions along the divisor $D$. While log Gromov--Witten invariants are usually difficult to compute, the large $\vec r$ limit of Theorem \ref{thm-mirror-intro} provides an effective way to compute Gromov--Witten invariants of $X_{D,\infty}$. Therefore, instead of computing the log Gromov--Witten invariant $N_\beta(\textbf{q}_{(d_1,\ldots, d_n)})$, we will compute the corresponding orbifold Gromov--Witten invariant of the infinite root stack $X_{D,\infty}$ with the same data. In other words, we compute genus zero invariants of $X_{D,\infty}$ of degree $\beta$ with $(d+1)$ markings, where $d=D\cdot \beta$, such that there are $d_i$, where $d_i=D_i\cdot \beta$, relative markings of contact order $1$ with $D_i$ and one interior marking with insertion $[pt]\bar{\psi}^{d-2}$. We will denote such invariants as $N^{\on{orb}}_\beta(\textbf{q}_{(d_1,\ldots, d_n)})$. Note that this is opposite to Section \ref{sec:local-log-orb} where relative markings are with maximal tangency instead of ``minimal'' tangency.

 \begin{assu}\label{assumption}
 Let $X$ be a Fano variety and $D=D_1+\cdots+D_n\in |-K_X|$. We assume that $D_i$ are nef for $1\leq i \leq n$ and satisfy the condition:
\[
\#\{i\in \{1,\ldots, n\}| D_i\cdot \beta>0\}\geq 2
\]
for all $\beta\in \on{NE}(X)$ with $D\cdot \beta\geq 2$. 
\end{assu}

When $t=0$, $X$ is Fano and $D=D_1+\cdots+D_n\in |-K_X|$, Assumption \ref{assumption} is the condition that the mirror map is trivial for the mirror theorem of $X_{D,\infty}$.

\begin{lemma}
Let $X$ be a Fano variety and $D=D_1+\cdots+D_n\in |-K_X|$ satisfies Assumption \ref{assumption}. Then the mirror map is trivial. More precisely, the $z^0$-coefficient of $I_{X_{D,\infty}}(Q,x,0,z)$ is $\sum_{i,j} x_{ij}[\textbf 1]_{(0,\ldots,0,j,0,\ldots,0)}$ and coefficients of positive powers of $z$ are zero (except for the term $z$):
\[
I_{X_{D,\infty}}(Q,x,0,z)=z+\sum_{i,j} x_{ij}[\textbf 1]_{(0,\ldots,0,j,0,\ldots,0)}+o(z^{-1}).
\]
\end{lemma}
\begin{proof}
First, we consider the non-extended $I$-function $I_{X_{D,\infty}}(Q,0,0,z)$:
\begin{align}\label{I-snc-t0}
\sum_{\beta\in \on{NE}(X)} J_{X, \beta}(0,z)Q^{\beta}
\prod_{i=1}^n\prod_{0< a< d_i}(D_i+az)[\textbf{1}]_{ (-d_1,-d_2,\ldots, -d_n)}.
\end{align}
Recall that, when $\beta=0$, we have
\[
J_{X, 0}(0,z)=z.
\]
When $\beta\neq 0$, we have
\[
J_{X, \beta}(0,z)=\sum_\alpha\left\langle\psi^{a-2}\phi_\alpha\right\rangle_{0,1,\beta}^X\phi^\alpha\left(\frac{1}{z}\right)^{a-1}
\]
and 
\[
a=\dim_{\mathbb C} X-K_X\cdot \beta-\deg (\phi_\alpha)\geq -K_X\cdot\beta.
\]
Then, one can try to extract the coefficients of highest power of $z$ in (\ref{I-snc-t0}). Assumption \ref{assumption} exactly implies that $I_{X_{D,\infty}}(Q,0,0,z)$ does not contain positive power of $z$ and the coefficient of $z^0$ is $0$. With the extended data given by $x_{ij}$, the $z^0$-coefficient becomes
\[
\sum_{i,j} x_{ij}[\textbf 1]_{(0,\ldots,0,j,0,\ldots,0)}
\]
which corresponds to insertions of extra relative markings. This completes the proof.
\end{proof}

\begin{proof}[Proof of Theorem \ref{thm:period}]

Recall that $d:=D\cdot \beta$ and $d_i:=D_i\cdot \beta$. We consider the extended $I$-function of $X_{D,\infty}$ that takes value in $H^*(X)$:

\begin{align*}
I_{X_{D,\infty},0}(Q,x,t,z):=\sum_{\substack{\beta\in \on{NE}(X),(k_{i1},\ldots,k_{im})\in (\mathbb Z_{\geq 0})^m\\ \sum_{j=1}^m jk_{ij}= d_i, 1\leq i \leq n}  } & J_{X, \beta}(t,z)Q^{\beta}\frac{\prod_{i=1}^n\prod_{j=1}^m x_{ij}^{k_{ij}}}{z^{\sum_{i=1}^n\sum_{j=1}^m k_{ij}}\prod_{i=1}^n\prod_{j=1}^m(k_{ij}!)}\\
& \cdot\left(\prod_{i=1}^n\prod_{0<a\leq d_i}(D_i+az)\right).
\end{align*}
We choose the extended data such that $x_{ij}=0$ for $j> 1$. It means that we only consider relative markings with contact order $1$. We write $x_i:=x_{i1}$. Then the extended $I$-function that takes value in $H^*(X)$ can be written as  
\begin{align*}
I_{X_{D,\infty},0}(Q,x,t,z):=\sum_{\beta\in \on{NE}(X)} & J_{X, \beta}(t,z)Q^{\beta}\frac{\prod_{i=1}^n x_{i}^{d_i}}{z^{\sum_{i=1}^n d_i}\prod_{i=1}^n(d_i!)}\cdot\left(\prod_{i=1}^n\prod_{0<a\leq d_i}(D_i+az)\right).
\end{align*}

Let $t=0$ and $d\geq 2$. Taking the coefficient of $\left(\prod_{i=1}^nx_i^{d_i}\right)\textbf{1}\in H^0(X)$, we have
\[
\sum_{\beta\in \on{NE}(X)}  \left[J_{X, \beta}(0,z)\right]_0 Q^{\beta},
\]
where
\[
\left[J_{X, \beta}(0,z)\right]_0=\int_{[\bM_{0,1}(X,\beta)]^{\on{vir}}}\psi^{d-2}\on{ev}^*([pt])\left(\frac{1}{z}\right)^{d-1}.
\]
The corresponding coefficient of the $J$-function of $X_{D,\infty}$ is
\[
\sum_{\beta\in \on{NE}(X)} \frac{1}{\prod_{i=1}^n (d_i!)}N^{\on{orb}}_\beta(\textbf{q}_{(d_1,\ldots,d_n)})Q^\beta\left(\frac{1}{z}\right)^{d-1}.
\]
Since the mirror map is trivial, we have the relative $I$-function equals to the relative $J$-function. Multiplying both sides by $d!$, we have 
\[
d!\int_{[\bM_{0,1}(X,\beta)]^{\on{vir}}}\psi^{d-2}\on{ev}^*([pt])=\frac{d!}{d_1!\cdots d_n!}N^{\on{orb}}_\beta(\textbf{q}_{(d_1,\ldots,d_n)}),
\]
where $d=\sum_{i=1}^n d_i=-K_X\cdot \beta$.

This yields the identity between regularized quantum period of $X$ and the classical period of the mirror superpotential $W$:
\[
\hat{G}_X(t)=c_{W,d,0},
\]
where we use orbifold Gromov--Witten invariants of $X_{D,\infty}$ instead of log Gromov--Witten invariants of $(X,D)$ on the right-hand side.
\end{proof}

\begin{rmk}
Similar computation can also be done if we do not assume that mirror maps are trivial. For example, we can take $D$ to be a smooth anticanonical divisor. Then one can apply the mirror theorem for smooth pairs in \cite{FTY} and run the same computation as in the proof of Theorem \ref{thm:period}. In this case, the mirror map is not trivial. Then relative Gromov--Witten invariants of $(X,D)$ are related to the regularized quantum periods via mirror maps. One may understand the difference as follows. For a log Calabi--Yau pair $(X,D)$, the boundary has to be sufficiently degenerate (a "large complex structure limit") in order to obtain a well-behaved mirror (non-singular, with the correct dimension ect.).
\end{rmk}


\bibliographystyle{amsxport}


\end{document}